\documentclass{amsart}
\usepackage{graphicx} 

\usepackage{amssymb}
\usepackage{amscd}

\theoremstyle{plain}
\newtheorem{thm}{Theorem}[section]
\newtheorem{cor}[thm]{Corollary}
\newtheorem{prop}[thm]{Proposition}

\title[Pure cactus group of degree 3 and configuration space of 4 points]{A note on the pure cactus group of degree three and the configuration space of four points on the circle.}

\author{Takatoshi Hama}
\address{Graduate School of Integrated Basic Sciences, Nihon University,
3-25-40 Sakurajosui, Setagaya-ku, Tokyo 156-8550, Japan}
\email{chta23018@g.nihon-u.ac.jp}

\author{Kazuhiro Ichihara}
\address{Department of Mathematics, 
College of Humanities and Sciences, Nihon University,
3-25-40 Sakurajosui, Setagaya-ku, Tokyo 156-8550, Japan}
\email{ichihara.kazuhiro@nihon-u.ac.jp}

\begin{document}



\begin{abstract}
The cactus group was introduced by Henriques and Kamnitzer as an analogue of the braid group. 
In this note, we provide an explicit description of the relationship between the pure cactus group of degree three and the configuration space of four points on the circle. 
\end{abstract}

\maketitle

\section{Introduction}

It is known that the braid group acts naturally on multiple tensor products in braided categories. 
As an analogue of it for coboundary categories, the \textit{cactus group} $J_n$ was introduced in \cite{HENRIQUES-KAMNITZER} motivated by the study of Kashiwara crystals. 
See the next section for the (purely algebraic) definition of $J_n$ that we use in this paper. 

Also in \cite[Theorem 9]{HENRIQUES-KAMNITZER}, it is shown that the cactus group $J_n$ surjects onto the symmetric group $S_n$ of degree $n$, and its kernel, called the \textit{pure cactus group} $PJ_n$, is isomorphic to the fundamental group of the Deligne-Mumford compactification  $\overline{M}_0^{n+1}(\mathbb{R})$ 
of the moduli space of real genus 0 curves over $\mathbb{R}$ with $n + 1$ marked points, which is closely related to the configuration space of $n+1$ points on the circle. 

In this paper, we provide an explicit description of the relationship between the pure cactus group $PJ_3$ of degree three and 
the compactification of the configuration space of four points on the circle. 

For this purpose, we use the Cayley complex ${C_3}^{\{2\}}$ of a certain subgroup ${J_3}^{\{2\}}$ of $J_3$ on which $PJ_3$ actually acts. 
This action is originally given in \cite{genevois2022cactusgroupsviewpointgeometric}. 
See the next section for the definition of the subgroup ${J_3}^{\{2\}}$. 

We also use a purely combinatorial compactification $\overline{X(4)}$ of the configuration space $X(4)$ of four distinct points on the circle introduced by M. Yoshida \cite{MYos96}. 
Also, see \cite{KNY99} and \cite{MN00} for relevant research. 

Then our main result is the following. 
 
\begin{thm}\label{thm}
Let $\widetilde{X(4)}$ be the universal cover of the compactification $\overline{X(4)}$ of the configuration space $X(4)$ of four points on the circle endowed with the action $\widetilde{\Gamma}$ of the fundamental group $\pi_1 (\overline{X(4)} )$ of $\overline{X(4)}$ as covering transformations. 
Let ${C_3}^{\{ 2\}}$ be the Cayley complex of the subgroup $J_3^{\{2\}}$ of the cactus group $J_3$. 
Then there exist an action $\Gamma$ of $PJ_3$ on ${C_3}^{\{2\}}$ and a bijective equivariant map $\varphi$ between $\widetilde{X(4)}$ and ${C_3}^{\{2\}}$ with respect to $\widetilde{\Gamma}$ and $\Gamma$, 
i.e., for any $g \in PJ_3$, there exists $\tilde{g} \in \pi_1 (\overline{X(4)} )$ such that the following diagram commutes;
\[
\begin{CD}
\widetilde{X(4)} @>{\widetilde{\Gamma}_{\tilde{g}}}>>\widetilde{X(4)}\\
\varphi@VVV   @VVV\varphi \\
{C_3}^{\{2\}} @>{\Gamma_g}>> {C_3}^{\{2\}}
\end{CD}
\]
\end{thm}

It then implies the following. 

\begin{cor}\label{cor}
The pure cactus group $PJ_3$ of degree three is isomorphic to the fundamental group $\pi_1 (\overline{X(4)} )$ of the compactification $\overline{X(4)}$ of the configuration space of four points on the circle. 
\end{cor}

Actually, it is known that $PJ_3$ is isomorphic to $\mathbb{Z}$ (see \cite{BCL24} for example) and also that the compactification $\overline{X(4)}$ of $X(4)$ is homeomorphic to the circle $S^1$ \cite{MYos96}. 
Thus, the last assertion is readily trivial. 
However, our constructions of the action of $PJ_3$ and the bijective equivariant map $\varphi$ can be regarded as an instructive example to study such actions of the pure cactus groups of higher degrees.

\section{Definitions}
\subsection{Cactus group}\label{subsec21}
    For a given integer $n\geq2$, the cactus group of degree $n$, denoted by $J_n$, 
    can be presented by the generators $s_{p,q}$ with $1\leq p < q \leq n$ and following relations.
\begin{itemize}
   \item $s_{p,q}^2 = 1$ for every $1 \leq p < q \leq n$,
   \item $s_{p,q}s_{m,r} = s_{m,r}s_{p,q}$ for all $1 \leq p < q \leq n$ and $1 \leq m < r \leq n$ satisfying $[p, q] \cap [m, r] = \varnothing$,
   \item $s_{p,q}s_{m,r} = s_{p+q-r,p+q-m}s_{p,q}$ for all $1 \leq p < q \leq n$ and $1 \leq m < r \leq n$ satisfying $[m, r] \subset [p, q]$.
\end{itemize}
Here, $[p,q]$ denotes the set $\{ p , p+1, \dots, q-1, q\}$ of integers for positive integers $p,q$ with $p<q$. 

    The cactus group $J_n$ admits a natural projection onto the symmetric group $S_n$ of degree $n$, and its kernel is called the \textit{pure cactus group} of degree $n$, denoted by $PJ_n$. 
    See \cite[Subsection 3.1]{HENRIQUES-KAMNITZER} or \cite[Section 1]{genevois2022cactusgroupsviewpointgeometric} for more details. 
    In particular, by \cite[Corollary 5.3]{BCL24}, $PJ_3$ has the following presentation.
\begin{align*}
  \langle \;(s_{12}s_{13})^{3} \;\rangle 
\end{align*}
That is, $PJ_3$ is the free abelian subgroup generated by $(s_{12}s_{13})^{3}$ in $J_3$. 
    
    For each integer $n\geq2$ and subset $S\subset[2, n]$, let ${J_n}^S$ be the subgroup of the cactus group $J_n$ of degree $n$ generated by the elements $s_{p,q}$ for $1 \leq p < q \leq n$ and $q-p+1\in S $ and defined by the following relations.
\begin{itemize}
   \item $s_{p,q}^2 = 1$ for every $1 \leq p < q \leq n$ satisfying $q-p+1\in S$,
   \item $s_{p,q}s_{m,r} = s_{m,r}s_{p,q}$ for every $1 \leq p < q \leq n$ and $1 \leq m < r \leq n$ satisfying $[p, q] \cap [m, r] = \varnothing$ and $q-p+1\in S$,
   \item $s_{p,q}s_{m,r} = s_{p+q-r,p+q-m}s_{p,q}$ for every $1 \leq p < q \leq n$ and $1 \leq m < r \leq n$ satisfying $[m, r] \subset [p, q]$ and $q-p+1\in S$.
\end{itemize}
    We denote the Cayley complex of the cactus group $J_n$
    and the subgroup ${J_n}^S$ by $C_n$ and ${C_n}^S$, respectively. 
    See \cite[Section 5]{genevois2022cactusgroupsviewpointgeometric} for more details. 

In particular, $J_3$ and $J_3^{\{2\}}$ admit the following presentations: 
\begin{align*}
    J_3 \ =& \langle s_{12}, s_{23}, s_{13}\mid {s_{12}}^2={s_{23}}^2={s_{13}}^2=1,  s_{12}s_{13}=s_{13}s_{23},s_{23}s_{13}=s_{13}s_{12} \rangle\\
    = & \langle s_{12}, s_{13}\mid {s_{12}}^2={s_{13}}^2=1\rangle \cong \mathbb{Z}_2 * \mathbb{Z}_2\\
    J_3^{\{2\}} =& \langle s_{12}, s_{23}\mid {s_{12}}^2={s_{23}}^2=1\rangle \cong J_3
\end{align*}
See \cite[Appendix A]{BCL24} for mote details. 


\subsection{Configuration space}
    For an integer $n\geq3$, the configuration space $X(n)$ of $n$ distinct points on the circle (regarded as the real projective line $P^1$) is defined by 
    \begin{align*}
    X(n)=PGL(2)\backslash\{(P^1)^n-\Delta\}   
    \end{align*}
    where $\Delta=\{(x_1,\cdots, x_n) \mid x_i = x_j\ \text{for some} \ i \neq j\ \}$ and the projective general liner group $PGL(2)$ acts diagonally and freely. 
    Then, the following natural purely combinatorial compactification $\overline{X(4)}$ of $X(4)$ was introduced in \cite{MYos96} and which we consider in this paper. 
    
    First, note that $X(4)$ has just three connected components, those are actually (1-dimensional) open intervals, coded by the figures in Figure~\ref{fig.1}. 
    We denote them by $[213]$, $[123]$, and $[132]$, respectively, as in the figure. 
    \begin{figure}[htb]
        \centering
        \includegraphics[width=0.9\linewidth]{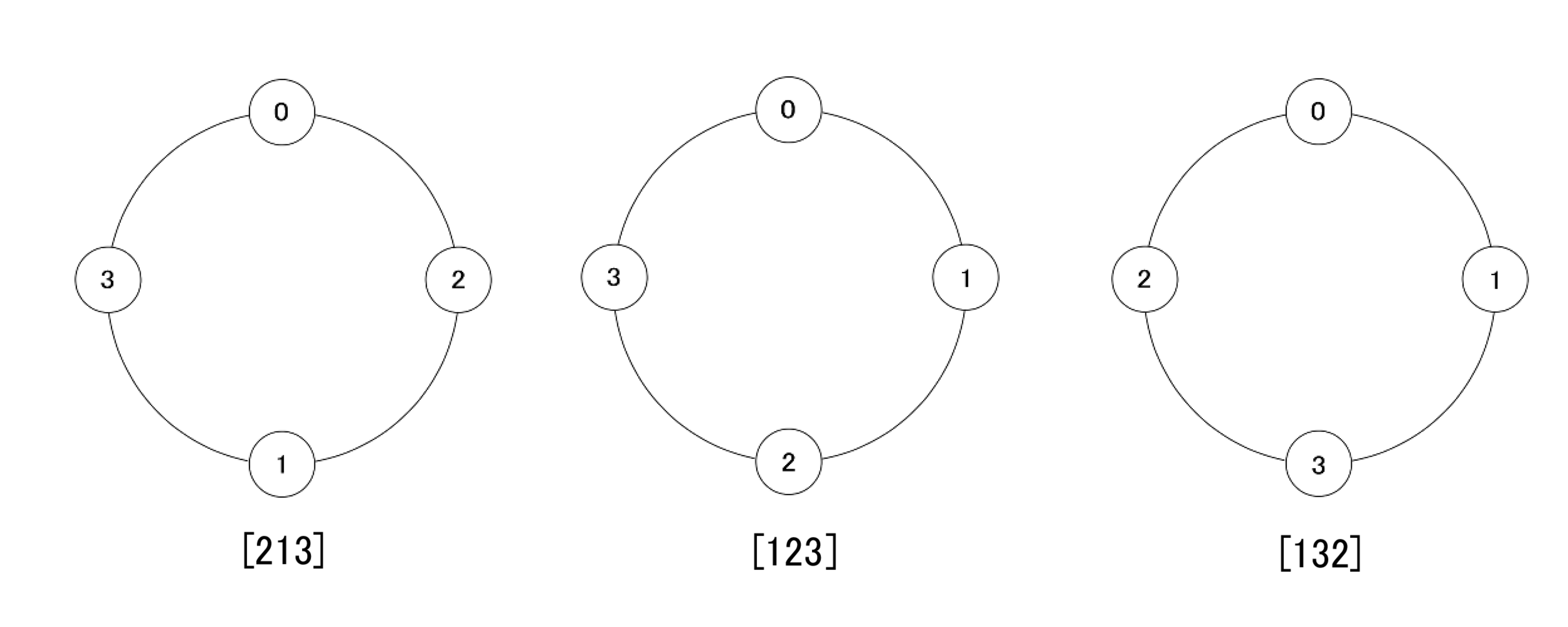}
        \caption{Representatives of three connected components of $X(3)$}
        \label{fig.1}
    \end{figure}

    Then, by regarding $X(4)$ is isomorphic to $P^1 - \{\text{three distinct points} \}$, 
we can uniquely compactify $X(4)$ by adding three points that represent configurations where exactly two adjacent points coincide. 
    We denote this compactification by $\overline{X(4)}$, which is naturally endowed with a complex structure of cells. 
    Thus, it can be regarded with a $1-$dimensional cell complex that is homeomorphic to the circle $S^1$. 
    In this paper, we consider the dual cell complex structure of the compactification $\overline{X(4)}$, which is visualized in Figure~\ref{fig:enter-label}.
    \begin{figure}[htb]
        \centering
        \includegraphics[width=0.6\linewidth]{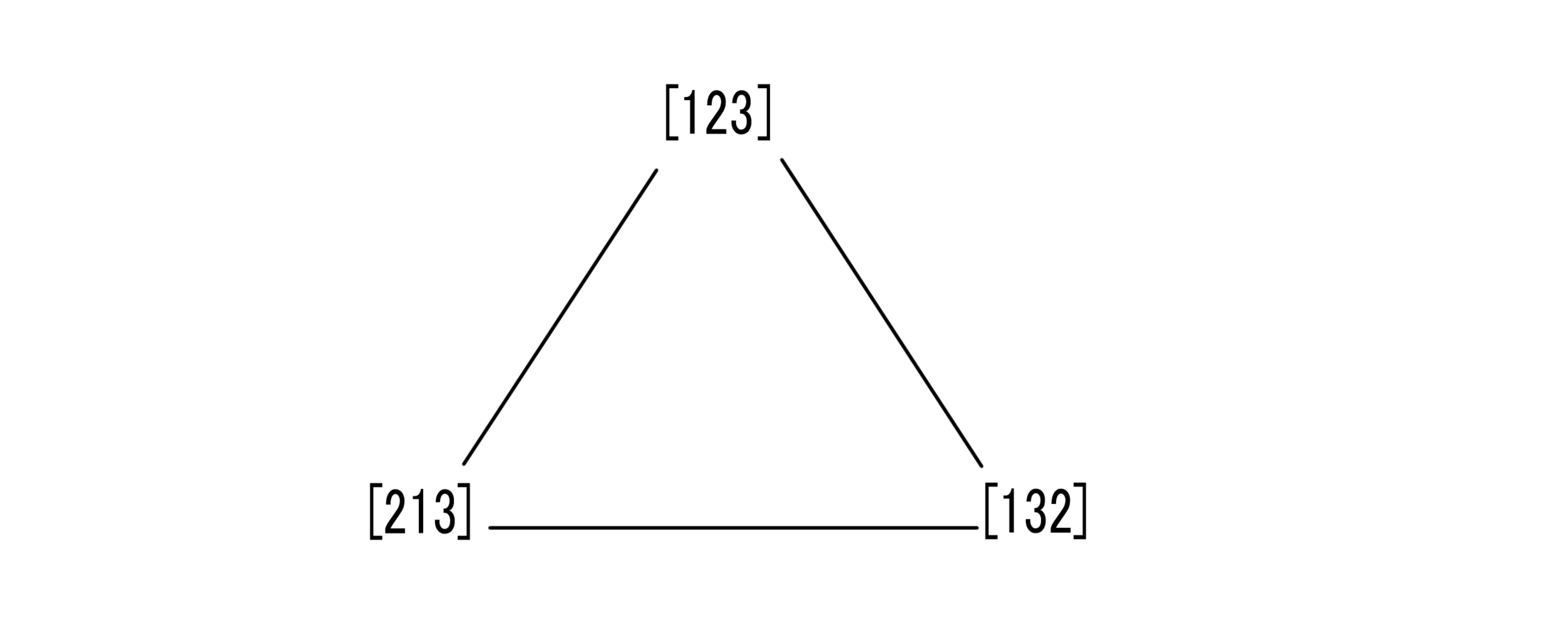}
        \caption{Dual cell complex structure of $\overline{X(4)}$}
        \label{fig:enter-label}
    \end{figure}

Then we see that the fundamental group $\pi_1 (\overline{X(4)})$ is isomorphic to $\mathbb{Z} \cong \langle \lambda \rangle$ generated by the element $\lambda$ represented by the loop given as $[123] \to [132] \to [213] \to [123]$. 

Now we consider the universal covering space $\widetilde{X(4)}$ of $\overline{X(4)}$. 
It naturally admits a cell complex structure lifted from that of $\overline{X(4)}$. 
Then the $0$-skeleton (vertices) are coded as 
    \begin{align*}
        \left( \widetilde{X(4)} \right)^{(0)} =\left\{ \; [a,b,c]_k \mid 
        (a,b,c) \in \{(1,3,2), (1,2,3), (2,1,3)\} \; \right\} , 
    \end{align*}
and $\widetilde{X(4)}$ can be expressed as the following: 
\begin{equation*}
\cdots \text{ --- } [123]_{-1} \text{ --- } [132]_{-1} \text{ --- } [213]_{0} \text{ --- } [123]_{0}  \text{ --- } [132]_{0}  \text{ --- } [213]_{1} \text{ --- } [123]_{1} \text{ --- } \cdots 
\end{equation*}

Under this setting, the fundamental group $\pi_1 (\overline{X(4)}) \cong \langle \lambda \rangle$ acts on $\widetilde{X(4)}$ as the covering transformations obtained by the natural extension of 
\begin{align*}
\widetilde{\Gamma}_0: \pi_1 (\overline{X(4)}) \times \left( \widetilde{X(4)} \right)^{(0)} & \longrightarrow \left( \widetilde{X(4)} \right)^{(0)} 
\end{align*}
defined by 
\begin{align*}
\widetilde{\Gamma}_{0} ( \lambda^k , [a,b,c]_n ) = \left( \widetilde{\Gamma}_{0} \right)_{ \lambda^k } ([a,b,c]_n):=[a,b,c]_{n+k}
\end{align*}
for an element $\lambda^k$ of $\pi_1 (\overline{X(4)})$ and an element $[a,b,c]_n$ of $\left( \widetilde{X(4)} \right)^{(0)}$ with integers $n$, $k$ and $(a, b, c) \in \{(2,1,3),(1,2,3),(1,3,2)\}$. 
  
        
\section{Action of $PJ_3$ on $C_3^{\{ 2 \}}$}

In \cite{genevois2022cactusgroupsviewpointgeometric}, an action of $PJ_3$ on $\left( {C_3}^{\{2\}} \right)^{(0)}$ given in the next proposition is essentially obtained. 

\begin{prop}\label{prop}
Let $\left( {C_3}^{\{2\}} \right)^{(0)}$ denote the 0-skeleton of the Cayley complex ${C_3}^{\{2\}}$, that is actually identified with $J_3^{\{2\}}$. 
     The map 
    \begin{align*}
        \Gamma_0: PJ_3 \times \left( {C_3}^{\{2\}} \right)^{(0)}  &\longrightarrow \left( {C_3}^{\{2\}} \right)^{(0)} 
    \end{align*}   
    defined by 
        \begin{align*}
        (g, h) \mapsto (\Gamma_{0} )_{g} (h):=
        \begin{cases}    
            gh & \text{if $gh$ is an element of ${J_3}^{\{2\}}$} \\
            ghs_{13} & \text{if $gh$ is not an element of ${J_3}^{\{2\}}$}
        \end{cases}
    \end{align*}
    induces a group action of $PJ_3$ on $\left( {C_3}^{\{2\}} \right)^{(0)}$. 
\end{prop}

\begin{proof}
As seen in Subsection~\ref{subsec21}, we can describe each element $h$ in $J_3^{\{2\}}$ with generators $s_{12}, s_{23}$ and integer $m$ as follows:
    \begin{align}\label{eq31}
        h= 
        \begin{cases}
        (s_{12}s_{23})^{\frac{m}{2}} & \text{if $m$ is even}  \\(s_{12}s_{23})^{\frac{m-1}{2}}s_{12} & \text{if $m$ is odd}
        \end{cases}
    \end{align}
 On the other hand, as seen in Subsection~\ref{subsec21}, 
 $PJ_3$ has the following presentation.
 \begin{align*}
   \langle \;(s_{12}s_{13})^{3} \;\rangle 
 \end{align*}
Thus every $g$ in $PJ_3$ can be presented as $g=(s_{12}s_{13})^{3k}$ for some integer $k$. 
This element $g$ can be rewritten with generators $s_{12}$ and $s_{23}$ using the relations of $J_3$ so that: 
    \begin{align*}
        g=&
         \begin{cases}       (s_{12}s_{13}s_{12}s_{13})(s_{12}s_{13}s_{12}s_{13})\cdots (s_{12}s_{13}s_{12}s_{13})  & \text{if $k$ is even}    \\(s_{12}s_{13}s_{12}s_{13})(s_{12}s_{13}s_{12}s_{13})\cdots(s_{12}s_{13}s_{12}s_{13})s_{12}s_{13} & \text{if $k$ is odd}
        \end{cases}\\
        =&
        \begin{cases}       (s_{12}s_{23}{s_{13}}^2)(s_{12}s_{23}{s_{13}}^2)\cdots (s_{12}s_{23}{s_{13}}^2)  & \text{if $k$ is even}    \\(s_{12}s_{23}{s_{13}}^2)(s_{12}s_{23}{s_{13}}^2)
        \cdots(s_{12}s_{23}{s_{13}}^2)s_{12}s_{13} & \text{if $k$ is odd}
        \end{cases}\\
        =&
        \begin{cases}(s_{12}s_{23})^{\frac{3k}{2}} & \text{if $k$ is even}    \\(s_{12}s_{23})^{\frac{3k-1}{2}}s_{12}s_{13} & \text{if $k$ is odd}
        \end{cases}
    \end{align*}

Here, in order to show that $\Gamma_0$ induces a group action of $PJ_3$ on $\left( {C_3}^{\{2\}} \right)^{(0)}$, it suffices to confirm that, for all $g=(s_{12}s_{13})^{3k}$ in $PJ_3$ and all $h$ in $J_3^{\{2\}}$, $(\Gamma_{0})_{g}$ shifts $h$ to just $3k$ next vertex.

In fact, by simple calculation, the following holds for every $g=(s_{12}s_{13})^{3k}$ in $PJ_3$ and every $h$ in $J_3^{\{2\}}$.
\begin{align*}
gh=
\begin{cases}
(s_{12}s_{23})^{\frac{3k+m}{2}} & \text{if $k$, $m$ are even} \\
(s_{12}s_{23})^{\frac{3k+m-1}{2}}s_{12}s_{13} &\text{if $k$ is odd, and $m$ is even}\\
(s_{12}s_{23})^{\frac{3k+m-1}{2}}s_{12} &\text{if $k$ is even, and $m$ is odd}\\
(s_{12}s_{23})^{\frac{3k+m}{2}}s_{13} &\text{if $k$, $m$ are odd}
\end{cases} \\
= \begin{cases}
(s_{12}s_{23})^{\frac{3k+m}{2}} & \text{if $k$, $m$ are even} \\
(s_{12}s_{23})^{\frac{3k+m-1}{2}}s_{12}s_{13} &\text{if $k$ is odd, and $m$ is even}\\
(s_{12}s_{23})^{\frac{3k+m-1}{2}}s_{12} &\text{if $k$ is even, and $m$ is odd}\\
(s_{12}s_{23})^{\frac{3k+m}{2}}s_{13} &\text{if $k$, $m$ are odd}
\end{cases}
\end{align*}
Here the element $h$ is assumed to be presented as \eqref{eq31}. 

Because $s_{13}$ is not an element in $J_3^{\{2\}}$, we can see that $(s_{12}s_{23})^{\frac{3k+m-1}{2}}s_{12}s_{13}$ and $(s_{12}s_{23})^{\frac{3k+m}{2}}s_{13}$ are not elements in $J_3^{\{2\}}$. 
Thus, we obtain the following. 
\begin{align*}
    \left( \Gamma_{0} \right)_{g} (h) =
\begin{cases}
(s_{12}s_{23})^{\frac{3k+m}{2}} & \text{if $k$, $m$ are even} \\
(s_{12}s_{23})^{\frac{3k+m-1}{2}}s_{12} &\text{if $k$ is odd, and $m$ is even}\\
(s_{12}s_{23})^{\frac{3k+m-1}{2}}s_{12} &\text{if $k$ is even, and $m$ is odd}\\
(s_{12}s_{23})^{\frac{3k+m}{2}} &\text{if $k$, $m$ are odd}
\end{cases}
\end{align*}
The word length of each element above is $3k+m = \lvert h \rvert +3k$, where $\lvert h \rvert$ denotes the word length of the element $h$. 
This implies that $(\Gamma_{0})_{g}$ shifts $h$ to just $3k$ next vertex for all $(s_{12}s_{13})^{3k}$ in $PJ_3$ and all $h$ in $J_3^{\{2\}}$ as desired. 
\end{proof}

Note that a group action $\Gamma$ of $PJ_3$ on ${C_3}^{\{2\}}$ can be obtained by extending the action $\Gamma_0$ naturally.

\section{Proofs}

In this section, we give proofs of Theorem~\ref{thm} and Corollary~\ref{cor}. 

\begin{proof}[Proof of Theorem~\ref{thm}]
Let $\varphi$ be the map from $\widetilde{X(4)}$ to $C_{3}^{\{2\}}$ 
obtained by naturally extending the map 
\begin{align*}
\varphi_0 : \left( \widetilde{X(4)} \right)^{(0)} & \longrightarrow \left( C_{3}^{\{2\}} \right)^{(0)} = J_3^{\{2\}}
\end{align*}
defined by the following: 
\begin{align*}
[213]_k &\longmapsto 
\begin{cases}
(s_{12}s_{23})^{\frac{3k}{2}}s_{12} & \text{if $k$ is even} \\
(s_{12}s_{23})^{\frac{3k+1}{2}} &\text{if $k$ is odd}
\end{cases}\\
[123]_k &\longmapsto 
\begin{cases}
(s_{12}s_{23})^{\frac{3k}{2}} & \text{if $k$ is even} \\
(s_{12}s_{23})^{\frac{3k-1}{2}}s_{12} &\text{if $k$ is odd}
\end{cases} \\
[132]_k &\longmapsto
\begin{cases}
(s_{12}s_{23})^{\frac{3k-2}{2}}s_{12} & \text{if $k$ is even} \\
(s_{12}s_{23})^{\frac{3k-1}{2}} &\text{if $k$ is odd}
\end{cases}
\end{align*}

Let $\Gamma_0$ be the actions of $PJ_3$ on $\left( C_{3}^{\{2\}} \right)^{(0)} = J_3^{\{2\}}$ defined in the previous section, and $\widetilde{\Gamma}_0$ the actions of $\pi_1 (\overline{X(4)})$ on $\left( \widetilde{X(4)} \right)^{(0)}$ as covering transformations. 
Then, to prove the theorem, it suffices to show that $\varphi_0$ is bijective and, for any $g \in PJ_3$, there exists $\tilde{g} \in \pi_1 (\overline{X(4)} )$ such that the following diagram commutes;
\[
\begin{CD}
\left( \widetilde{X(4)} \right)^{(0)} @>{\left( \widetilde{\Gamma_0} \right)_{\tilde{g}}}>> \left( \widetilde{X(4)} \right)^{(0)}\\
\varphi_0@VVV   @VVV\varphi_0 \\
J_3^{\{2\}} @>{\left( \Gamma_0 \right)_g }>> J_3^{\{2\}}
\end{CD}
\]

We first show that $\varphi_0$ is bijective. 
By construction of $\varphi_0$, together with the presentation of $J_3^{\{2\}}$ given in Subsection~\ref{subsec21}, it is clear that $\varphi_0$ is injective.
For an element $g$ in $J_3^{\{2\}} = \left(C_{3}^{\{2\}} \right)^{(0)} $, if the word length $\lvert g \rvert$ of $g$ is congruent to 1 modulo 3, we provide an element $[213]_{\frac{\lvert g \rvert-1}{3}}$ of $\left( \widetilde{X(4)} \right)^{(0)}$.
If the word length $\lvert g \rvert$ of $g$ is congruent to 0 modulo 3, we provide $[123]_{\frac{\lvert g \rvert}{3}}$.
If the word length $\lvert g \rvert$ of $g$ is congruent to $-1$ modulo 3, we provide $[132]_{\frac{\lvert g \rvert+1}{3}}$.
\noindent
Then, we obtain the following. 
\begin{align*}
\varphi_0([213]_{\frac{\lvert g \rvert-1}{3}})=&
\begin{cases}
(s_{12}s_{23})^{\frac{3\frac{\lvert g \rvert-1}{3}}{2}}s_{12} & \text{if $\frac{\lvert g \rvert-1}{3}$ is even} \\
(s_{12}s_{23})^{\frac{3\frac{\lvert g \rvert-1}{3}+1}{2}} &\text{if $\frac{\lvert g \rvert-1}{3}$ is odd}
\end{cases}\\
=&
\begin{cases}
(s_{12}s_{23})^{\frac{\lvert g \rvert-1}{2}}s_{12} & \text{if $\frac{\lvert g \rvert-1}{3}$ is even} \\
(s_{12}s_{23})^{\frac{\lvert g \rvert}{2}} &\text{if $\frac{\lvert g \rvert-1}{3}$ is odd}
\end{cases} \\
\varphi_0([123]_{\frac{\lvert g \rvert}{3}})=&
\begin{cases}
(s_{12}s_{23})^{\frac{3\frac{\lvert g \rvert}{3}}{2}} & \text{if $\frac{\lvert g \rvert}{3}$ is even} \\
(s_{12}s_{23})^{\frac{3\frac{\lvert g \rvert}{3}-1}{2}}s_{12} &\text{if $\frac{\lvert g \rvert}{3}$ is odd}
\end{cases}\\
=&
\begin{cases}
(s_{12}s_{23})^{\frac{\lvert g \rvert}{2}} & \text{if $\frac{\lvert g \rvert}{3}$ is even} \\
(s_{12}s_{23})^{\frac{\lvert g \rvert-1}{2}}s_{12} &\text{if $\frac{\lvert g \rvert}{3}$ is odd}
\end{cases} \\
\varphi_0([132]_{\frac{\lvert g \rvert+1}{3}})=&
\begin{cases}
(s_{12}s_{23})^{\frac{3\frac{\lvert g \rvert+1}{3}-2}{2}}s_{12} & \text{if $\frac{\lvert g \rvert+1}{3}$ is even} \\
(s_{12}s_{23})^{\frac{3\frac{\lvert g \rvert+1}{3}-1}{2}} &\text{if $\frac{\lvert g \rvert+1}{3}$ is odd}
\end{cases}\\
=&
\begin{cases}
(s_{12}s_{23})^{\frac{\lvert g \rvert-1}{2}}s_{12} & \text{if $\frac{\lvert g \rvert+1}{3}$ is even} \\
(s_{12}s_{23})^{\frac{\lvert g \rvert}{2}} &\text{if $\frac{\lvert g \rvert+1}{3}$ is odd}
\end{cases}
\end{align*}
The above equations imply that we get the elements corresponding to the elements in ${J_3}^{\{2\}}$ for all word length determined by generators $s_{12}$, $s_{23}$. 
Therefore, $\varphi_0$ is surjective.

We next confirm that, for any $g \in PJ_3$, there exists $\tilde{g} \in \pi_1 (\overline{X(4)} )$ such that 
\begin{equation}\label{eq1}    \varphi_0\circ(\widetilde{\Gamma}_0)_{\tilde{g}}=(\Gamma_0)_g\circ\varphi_0
\end{equation}
holds. 

Let $g=(s_{12}s_{13})^{3j}$ be an element in $PJ_3$ with an integer $j$. 
Take the element $\lambda^j$ in $\pi_1 (\overline{X(4)} )$ as  $\tilde{g} \in \pi_1 (\overline{X(4)} ) $ corresponding to $g$.
Then, the following hold for any integer $i$. 
\begin{align*}
\varphi_0(\left(\widetilde{\Gamma}_0\right)_{\lambda^j}([213]_i))&=\varphi_0([213]_{i+j})\\
&=\begin{cases}
(s_{12}s_{23})^{\frac{3(i+j)}{2}}s_{12} & \text{if $i+j$ is even} \\
(s_{12}s_{23})^{\frac{3(i+j)+1}{2}} &\text{if $i+j$ is odd}
\end{cases}\\
\varphi_0(\left(\widetilde{\Gamma}_0\right)_{\lambda^j}([123]_i))&=\varphi_0([123]_{i+j})\\
&=
\begin{cases}
(s_{12}s_{23})^{\frac{3(i+j)}{2}} & \text{if $i+j$ is even} \\
(s_{12}s_{23})^{\frac{3(i+j)-1}{2}}s_{12} &\text{if $i+j$ is odd}
\end{cases}\\
\varphi_0(\left(\widetilde{\Gamma}_0\right)_{\lambda^j}([132]_i))&=\varphi_0([132]_{i+j})\\
&=
\begin{cases}
(s_{12}s_{23})^{\frac{3(i+j)-2}{2}}s_{12} & \text{if $i+j$ is even} \\
(s_{12}s_{23})^{\frac{3(i+j)-1}{2}} &\text{if $i+j$ is odd}
\end{cases}
\end{align*}

On the other hand, for $g=(s_{12}s_{13})^{3j}$ in $PJ_3$, we obtain the following. 
\begin{align*}
\left(\Gamma_0\right)_{(s_{12}s_{13})^{3j}}(\varphi_0 ([213]_i))
&=
\begin{cases}
\left( \Gamma_0 \right)_{{(s_{12}s_{13})^{3j}}}((s_{12}s_{23})^{\frac{3i}{2}}s_{12}) & \text{if $i$ is even} \\
\left( \Gamma_0 \right)_{{(s_{12}s_{13})^{3j}}}((s_{12}s_{23})^{\frac{3i+1}{2}}) &\text{if $i$ is odd}
\end{cases} \\
&=
\begin{cases}
(s_{12}s_{23})^{\frac{3(i+j)}{2}}s_{12} & \text{if $i$, $j$ are even} \\
(s_{12}s_{23})^{\frac{3(i+j)+1}{2}} &\text{if $i$ is even, and $j$ is odd}\\
(s_{12}s_{23})^{\frac{3(i+j)+1}{2}} &\text{if $i$ is odd, and $j$ is even}\\
(s_{12}s_{23})^{\frac{3(i+j)}{2}}s_{12} &\text{if $i$, $j$ are odd}
\end{cases}\\
&=
\begin{cases}
(s_{12}s_{23})^{\frac{3(i+j)}{2}}s_{12} & \text{if $i+j$ is even} \\
(s_{12}s_{23})^{\frac{3(i+j)+1}{2}} &\text{if $i+j$ is odd}
\end{cases}
\end{align*}

\begin{align*}
\left(\Gamma_0\right)_{(s_{12}s_{13})^{3j}}(\varphi_0 ([123]_i))
&=
\begin{cases}
\left(\Gamma_0\right)_{(s_{12}s_{13})^{3j}}((s_{12}s_{23})^{\frac{3i}{2}}) & \text{if $i$ is even}\\
\left(\Gamma_0\right)_{(s_{12}s_{13})^{3j}}((s_{12}s_{23})^{\frac{3i-1}{2}}s_{12}) &\text{if $i$ is odd}
\end{cases} \\
&=
\begin{cases}
(s_{12}s_{23})^{\frac{3(i+j)}{2}} & \text{if $i$, $j$ are even} \\
(s_{12}s_{23})^{\frac{3(i+j)-1}{2}}s_{12} &\text{if $i$ is even, and $j$ is odd}\\
(s_{12}s_{23})^{\frac{3(i+j)-1}{2}}s_{12} &\text{if $i$ is odd, and $j$ is even}\\
(s_{12}s_{23})^{\frac{3(i+j)}{2}} &\text{if $i$, $j$ are odd}
\end{cases}\\
&=
\begin{cases}
(s_{12}s_{23})^{\frac{3(i+j)}{2}} & \text{if $i+j$ is even} \\
(s_{12}s_{23})^{\frac{3(i+j)-1}{2}}s_{12} &\text{if $i+j$ is odd}
\end{cases}
\end{align*}

\begin{align*}
\left(\Gamma_0\right)_{(s_{12}s_{13})^{3j}}(\varphi_0 ([132]_i))
&=
\begin{cases}
\left(\Gamma_0\right)_{(s_{12}s_{13})^{3j}}((s_{12}s_{23})^{\frac{3i-2}{2}}s_{12}) & \text{if $i$ is even} \\
\left(\Gamma_0\right)_{(s_{12}s_{13})^{3j}}((s_{12}s_{23})^{\frac{3i-1}{2}}) &\text{if $i$ is odd}
\end{cases} \\
&=
\begin{cases}
(s_{12}s_{23})^{\frac{3(i+j)-2}{2}}s_{12} & \text{if $i$, $j$ are even} \\
(s_{12}s_{23})^{\frac{3(i+j)-1}{2}} &\text{if $i$ is even, $j$ is odd}\\
(s_{12}s_{23})^{\frac{3(i+j)-1}{2}} &\text{if $i$ is odd, $j$ is even}\\
(s_{12}s_{23})^{\frac{3(i+j)-2}{2}}s_{12} &\text{if $i$, $j$ are odd}
\end{cases}\\
&=
\begin{cases}
(s_{12}s_{23})^{\frac{3(i+j)-2}{2}}s_{12} & \text{if $i+j$ is even} \\
(s_{12}s_{23})^{\frac{3(i+j)-1}{2}} &\text{if $i+j$ is odd}
\end{cases}
\end{align*}
\noindent
Thus, the equation~\eqref{eq1} holds. 
\end{proof}

\bigskip

\begin{proof}[Proof of Corollary~\ref{cor}]

Let $h : PJ_3 \to \pi_1 (\overline{X(4)} )$ sending $g=(s_{12}s_{13})^{3j}$ in $PJ_3$ to $h(g) = \lambda^j$ in $\pi_1 (\overline{X(4)} )$. 
Then, due to Theorem~\ref{thm}, this $h$ is a homomorphism. 
Precisely, by the equation~\eqref{eq1}, together with Proposition~\ref{prop}, we have 
\begin{align*}
    \left(\widetilde{\Gamma}_0\right)_{h(g_1 g_2)}
     &=  \varphi_0^{-1} \circ \left( \Gamma_0 \right)_{g_1 g_2} \circ \varphi_0 \\
     &=  \varphi_0^{-1} \circ \left( \Gamma_0 \right)_{g_1} \circ \left( \Gamma_0 \right)_{g_2} \circ \varphi_0 \\
     &=  \varphi_0^{-1} \circ \left( \Gamma_0 \right)_{g_1} \circ \varphi_0 \circ \varphi_0^{-1} \circ \left( \Gamma_0 \right)_{g_2} \circ \varphi_0 \\
    &= \left(\widetilde{\Gamma}_0\right)_{h(g_1)} \circ \left(\widetilde{\Gamma}_0\right)_{h(g_2)}.
\end{align*}
This implies that $h$ is a homomorphism. 

To show that $h$ is an isomorphism, equivalently, $h$ is a bijection, it suffices to show that the map $h' : \pi_1 (\overline{X(4)}) \to PJ_3$ sending $\lambda^j$ in $\pi_1 (\overline{X(4)})$ to the element $(s_{12} s_{13} )^{3k} $ in $PJ_3$ is a homomorphism, as they satisfy $h \circ h' = h' \circ h$ is the identity map. 
It is done in the same way as for $h$ above. 
\end{proof}

\bibliographystyle{amsplain}
\bibliography{reference}
\end{document}